\documentclass[11pt,a4paper]{amsart}

\usepackage{amsfonts}

\usepackage[T1]{fontenc}

\usepackage{enumerate}

\usepackage[abbrev,backrefs]{amsrefs}

\usepackage[utf8]{inputenc}

\usepackage{graphicx}
\usepackage{amssymb}
\usepackage{esint}

\usepackage{palatino}

\renewcommand{\div}{\operatorname{div}}
\newcommand{\abs}[1]{\mathopen\lvert#1\mathclose\rvert}

\newcommand{\biggabs}[1]{\biggl\lvert#1\biggr\rvert}
\newcommand{\norm}[1]{\mathopen\lVert#1\mathclose\rVert}
\newcommand{\Bignorm}[1]{\mathopen\Big\lVert#1\mathclose\Big\rVert}

\newcommand{\N}{{\mathbb N}}
\newcommand{\R}{{\mathbb R}}

\newcommand{\cH}{\mathcal{H}}

\DeclareMathOperator{\supp}{supp}

\DeclareMathOperator{\diam}{diam}

\DeclareMathOperator{\sgn}{sgn}

\DeclareMathOperator{\capt}{cap}

\newcommand{\dif}{\,\mathrm{d}}

\newcommand{\loc}{_\mathrm{loc}}

\theoremstyle{plain}
\newtheorem{proposition}{Proposition}[section]
\newtheorem{lemma}[proposition]{Lemma}
\newtheorem{theorem}{Theorem}
\newtheorem{corollary}[proposition]{Corollary}

\theoremstyle{definition}

\theoremstyle{remark}

\newtheorem{claim}{Claim}

\usepackage[pdftitle={Strong maximum principle for Schroedinger operators with singular potentials},%
pdfauthor={Luigi Orsina and Augusto C. Ponce},%
backref=page]{hyperref}
\hypersetup{colorlinks=true, linkcolor=black, urlcolor=black, citecolor=black}

\numberwithin{equation}{section}

\title[Strong maximum principle for Schr\"{o}dinger operators]{Strong maximum principle for  Schr\"{o}dinger operators with singular potential}

\author{Luigi Orsina}
\address{
Luigi Orsina\hfill\break\indent
``Sapienza'', Universit\`a di Roma\hfill\break\indent
Dipartimento di Matematica \hfill\break\indent
P.le A.~Moro 2\hfill\break\indent
00185 Roma\hfill\break\indent
Italy}
\email{orsina@mat.uniroma1.it}

\author{Augusto C. Ponce}
\address{
Augusto C. Ponce\hfill\break\indent
 Université catholique de Louvain\hfill\break\indent
 Institut de recherche en mathématique et physique\hfill\break\indent
 Chemin du cyclotron 2, L7.01.02\hfill\break\indent
1348 Louvain-la-Neuve\hfill\break\indent
Belgium}
\email{Augusto.Ponce@uclouvain.be}



\begin{document}
\begin{abstract}
We prove that for every \(p > 1\) and for every potential \(V \in L^p\), any nonnegative function satisfying \(-\Delta u + V u \ge 0\) in an open connected set of \(\R^N\) is either identically zero or its level set \(\{u = 0\}\) has zero \(W^{2, p}\) capacity.
This gives an affirmative answer to an open problem of Bénilan and Brezis concerning a bridge between Serrin-Stampacchia's strong maximum principle for \(p > \frac{N}{2}\) and Ancona's strong maximum principle for \(p = 1\).
The proof is based on the construction of suitable test functions depending on the level set \(\{u = 0\}\), and on the existence of solutions of the Dirichlet problem for the Schrödinger operator with diffuse measure data.
\end{abstract}

\maketitle

\section{Introduction and main result}

We investigate the strong maximum principle for the Schr\"{o}dinger operator \(-\Delta + V\) where \(V : \Omega \to \R\) is a given potential and \(\Omega \subset \R^N\) is an open connected set.
More precisely, let \(u : \Omega \to \R\) be a nonnegative function satisfying
\begin{equation}
\label{eqPDEInequality}
- \Delta u + V u \ge 0
\quad \text{in \(\Omega\).} 
\end{equation}
Assuming that \(u\) vanishes somewhere in \(\Omega\), is it true that \(u\) vanishes identically in \(\Omega\)?
This is indeed the case when \(V = 0\), but in general the answer is negative.
For instance, the function \(u : \R^N \to \R\) defined by \(u(x) = \norm{x}^2\) satisfies
\[
- \Delta u + \frac{2N}{\norm{x}^2} u = 0 
\quad \text{in \(\R^N\).}
\] 
A similar example is given by the function \(u(x) = \norm{x}\); in this case the differential inequality \eqref{eqPDEInequality} holds in the sense of distributions in \(\R^N\).

In this paper, we provide a condition on the potential \(V\) and on the set where \(u\) vanishes which ensures that \(u\) equals zero in \(\Omega\).
Our main result is the following:

\begin{theorem}
\label{theoremStrongMaximumPrinciple}
Let \(\Omega \subset \R^N\) be an open connected set, \(p > 1\) and \(V \in L^p(\Omega)\).
If \(u \in L^1(\Omega)\) is a nonnegative function such that \(Vu \in L^1(\Omega)\) and
\[
- \Delta u + V u \ge 0
\quad \text{in the sense of distributions in \(\Omega\),}
\]
and if the average integral of \(u\) satisfies
\begin{equation}
\label{eqVanishingAverage}
\lim_{r \to 0}{\fint\limits_{B(x; r)} u} = 0
\end{equation}
for every point \(x\) in a compact subset of \(\Omega\) with positive \(W^{2, p}\) capacity, then \(u = 0\) almost everywhere in \(\Omega\).
\end{theorem}

Since \(u\) is nonnegative, the vanishing condition \eqref{eqVanishingAverage} identifies exactly the Lebesgue points of \(u\) where the precise representative of \(u\) vanishes.
By abuse of notation, we sometimes denote this set as \(\{u = 0\}\);
there is no ambiguity for instance when the function \(u\) is continuous.

The \(W^{2, p}\) capacity of a compact set \(K \subset \R^{N}\) is defined as
\[
\capt_{W^{2, p}}{(K)}
= \inf{\Big\{ \norm{\varphi}_{W^{2, p}(\R^{N})}^p : \varphi \in C_c^\infty(\R^{N})\ \text{nonnegative}\  \text{and}\ \varphi > 1\ \text{in \(K\)}  \Big\}}.
\]
This capacity has the same sets of positive capacity as the corresponding Bessel capacity by Calderón's isomorphism between \(W^{2, p}\) and \(L^p\) via Bessel potentials \citelist{ \cite{Stein:1970}*{Chapter~V, Theorem~3}  \cite{Adams_Hedberg:1996}*{Theorem~1.2.3}}.
By the relation between the Sobolev capacity and the Hausdorff measure~\cite{Adams_Hedberg:1996}*{Theorem~5.1.13}, we conclude that a nonnegative function satisfying \eqref{eqPDEInequality} is either almost everywhere zero or has a level set \(\{u = 0\}\) with Hausdorff dimension at most \(N - 2p\).

When \(p > \frac{N}{2}\), by the Morrey-Sobolev imbedding every singleton \(\{a\}\) has positive \(W^{2, p}\) capacity.
In this case, by Theorem~\ref{theoremStrongMaximumPrinciple} above we deduce that if \(u(a) = 0\) for some \(a \in \Omega\), then we have \(u = 0\) in \(\Omega\).
We then recover the strong maximum principle based on the Harnack inequality.
Such an inequality is obtained by a clever adaptation of Moser's iteration technique~\cites{Moser:1961}, and was implemented independently by Serrin~\cite{Serrin:1964}*{Theorem~5} and by Stampacchia~\cite{Stampacchia:1965}*{Corollaire~8.2} for solutions associated to the Schrödinger operator \(-\Delta + V\), and then by Trudinger~\cite{Trudinger:1973}*{Theorem~5.2} for supersolutions.

The counterpart of  Theorem~\ref{theoremStrongMaximumPrinciple} for \(p = 1\) and potentials \(V \in L^1(\Omega)\) is given in terms of the  --- Newtonian --- \(W^{1, 2}\) capacity.
This beautiful result  was originally proved by Ancona~\cite{Ancona:1979}*{Th\'{e}orème~9} using tools from Potential theory, and extends a unique continuation principle of Bénilan and Brezis~\cite{Benilan_Brezis:2004}*{Theorem~C.1} for nonnegative functions with compact support.
An alternative proof --- in the spirit of elliptic PDEs --- may be found in \cite{Brezis_Ponce:2003}; see also Section~\ref{sectionStrongMaximumPrincipleLebesgue} below.

Theorem~\ref{theoremStrongMaximumPrinciple} above gives an affirmative answer to a question raised by B\'{e}nilan and Brezis \cite{Benilan_Brezis:2004}*{Open problem~4} asking whether there would be a bridge between Serrin-Stampacchia's strong maximum principle for potentials \(V \in L^p(\Omega)\) with \(p > \frac{N}{2}\) and Ancona's strong maximum principle with \(p = 1\).
The link between Ancona's result and ours relies on the fact that the \(W^{1, 2}\) capacity may be seen as a limit of the \(W^{2, p}\) capacities as \(p\) tends to \(1\) \citelist{\cite{Brezis_Marcus_Ponce:2007}*{Theorem~4.E.1} \cite{Ponce:2013}*{Chapter~12}}; see also Section~\ref{sectionConverse} below.

The proof of Theorem~\ref{theoremStrongMaximumPrinciple} is based on a suitable choice of nonnegative test functions \(w\) for which we have the inequality
\[
\int\limits_\Omega u (-\Delta w + V w) \ge 0.
\]
By assumption this holds for test functions \(w \in C_c^\infty(\Omega)\).
We justify via an approximation procedure that  for every \(\epsilon > 0\) it is possible to choose \(w = w_\epsilon\) such that
\[
-\Delta w + V w
= \mu - \epsilon \chi_{A_\epsilon},
\]
where \(\mu\) is a positive measure supported by the set \(\{u = 0\}\), and \((A_\epsilon)_{\epsilon > 0}\) is a family of measurable subsets of \(\Omega\) such that the Lebesgue measure of \(\Omega \setminus A_\epsilon\) converges to zero as \(\epsilon\) tends to zero.
The assumption \(V \in L^p(\Omega)\) ensures the existence of solutions of this equation for any measure \(\mu\) which is diffuse with respect to the \(W^{2, p}\) capacity.
For a measure \(\mu\) supported by the set \(\{u = 0\}\), we have --- at least formally ---
\[
\int\limits_\Omega u \dif\mu = 0,
\]
and we deduce that, for every \(\epsilon > 0\),
\[
\epsilon \int\limits_{A_\epsilon} u \le 0.
\]
The conclusion follows as \(\epsilon\) tends to zero.
The tools needed to justify this argument are developed in Sections~\ref{sectionStrongMaximumPrincipleLebesgue}--\ref{sectionChoiceTestFunctions}.

In Section~\ref{sectionConverse} below we prove the following converse of Theorem~\ref{theoremStrongMaximumPrinciple}: for every compact set \(K \subset \Omega\) with zero \(W^{2, p}\) capacity there exist \(V \in L^p(\Omega)\) and a nonnegative smooth function \(u\) vanishing precisely on \(K\) such that \(-\Delta u + V u = 0\).
An adaptation of the proof also gives the counterpart for \(p = 1\) in terms of the \(W^{1, 2}\) capacity, which is also new in this context.
Our construction is motivated by de~la~Vallée~Poussin's interpretation of sets of zero capacity in terms of level sets \(\{w = +\infty\}\) of functions \(w\) with finite energy~\cite{DLVP:1932}*{\S 70}.

\section{A strong maximum principle in terms of the Lebesgue measure}
\label{sectionStrongMaximumPrincipleLebesgue}

One of the ingredients in the proof of Theorem~\ref{theoremStrongMaximumPrinciple} is a particular case of Ancona's strong maximum principle when the vanishing condition is stated in terms of the Lebesgue measure, which is enough in some applications~\cites{Benilan_Brezis:2004,Lucia:2005}; see also \cite{VanSchaftingen_Willem:2008}.
We present a sketch of the proof from \cite{Brezis_Ponce:2003} for the sake of completeness.

\begin{proposition}
\label{propositionStrongMaximumPrincipleLebesgue}
Let \(\Omega \subset \R^N\) be an open connected set and \(V \in L^1(\Omega)\).
If \(u \in W\loc^{1, 2}(\Omega)\) is a nonnegative function such that \(Vu \in L^1(\Omega)\) and
\[
- \Delta u + V u \ge 0
\quad \text{in the sense of distributions in \(\Omega\),}
\]
and if
\[
\lim_{r \to 0}{\fint\limits_{B(x; r)} u} = 0
\]
for every \(x\) in a subset of \(\Omega\) with positive Lebesgue measure, then \(u = 0\) almost everywhere in \(\Omega\).
\end{proposition}

\begin{proof}
For every \(\varphi \in C_c^\infty(\Omega)\), by an approximation argument we may use the test function \(\dfrac{\varphi^2}{1 + u}\) in the weak inequality satisfied by \(u\) to get
\[
\int\limits_\Omega \frac{\abs{\nabla u}^2}{(1 + u)^2} \varphi^{2}
\le 4 \int\limits_\Omega \abs{\nabla \varphi}^2 + 2\int\limits_\Omega V^+\varphi^{2}.
\] 

Given a connected open subset \(\omega \Subset \Omega\) such that \(u = 0\) in a subset of $\omega$ of positive Lebesgue measure, the function \(\log{(1 + u)}\) also vanishes in a subset of $\omega$ of positive measure, whence by the Poincar\'{e} inequality --- proved for example by a contradiction argument ---, we have
\[
\int\limits_\omega \abs{\log(1 + u)}^2
\le
C_1 \int\limits_\omega \abs{\nabla \log(1 + u)}^2
=
C_1 \int\limits_\omega \frac{\abs{\nabla u}^2}{(1 + u)^2}.
\]
Choosing \(\varphi\) such that \(\varphi = 1\) in \(\omega\), we deduce that
\[
\frac{1}{C_1} \int\limits_\omega \abs{\log(1 + u)}^2
\le
4 \int\limits_\Omega \abs{\nabla \varphi}^2 + 2\int\limits_\Omega V^+\varphi^{2}.
\]
In particular, the right-hand side does not depend on \(u\);
the constant \(C_1\) arising from the Poincar\'{e} inequality depends on the size of the level set \(\{u = 0\}\).
In view of the linear nature of the differential inequality satisfied by \(u\), the estimate above is thus invariant if we replace \(u\) by \(\frac{u}{\delta}\) for any \(\delta > 0\).
As \(\delta\) tends to zero, the function \(\log(1 + \frac{u}{\delta})\) diverges to infinity on the set \(\{u > 0\}\).
On the other hand, by the above estimate the functions \(\log(1 + \frac{u}{\delta})\) are bounded in \(L^2(\omega)\) independently of \(\delta\).
By Fatou's lemma, it follows that \(\{u > 0\}\) must have zero Lebesgue measure in \(\omega\).
\end{proof}

Compared with Theorem~\ref{theoremStrongMaximumPrinciple} we have assumed that \(u \in W\loc^{1, 2}(\Omega)\).
We now explain why this is not a restriction for establishing the strong maximum principle for merely \(L^1\) functions by using a truncation argument.
We first observe that since \(u\) is nonnegative, we have
\[
V u \le V^+ u, 
\]
so replacing \(V\) by \(V^+\) if necessary, we may assume from the beginning that the potential \(V\) is nonnegative.
Next, for every \(\kappa > 0\), the function \(\min{\{u, \kappa\}}\) is also a supersolution for the Schr\"{o}dinger operator \(-\Delta + V\).
This may be seen as a consequence of the following variant of Kato's inequality:

\begin{lemma}
\label{lemmaKato}
Let \(v \in L^1(\Omega)\) and \(f \in L^1(\Omega)\) be such that
\[
\Delta v 
\le f
\quad \text{in the sense of distributions in \(\Omega\).}
\]
Then, for every \(\kappa \in \R\), we have
\[
\Delta \min{\{v, \kappa\}} 
\le \chi_{\{v < \kappa\}} f
\quad \text{in the sense of distributions in \(\Omega\).}
\]
\end{lemma}

Here, \(\chi_A\) denotes the characteristic function of a set \(A \subset \R^N\).
Kato's inequality has been introduced by Kato to study Schrödinger operators with singular potentials \(V\).
Strictly speaking, Kato's inequality concerns functions \(v\) such that \(\Delta v \in L^1(\Omega)\)~\cite{Kato:1972}*{Lemma~A}.
This need not be true in our case since \(\Delta v\) may be a locally finite measure, but the proof can be performed in the same way by approximation~\citelist{\cite{Ponce:2012}*{Propositions~5.7 and 5.9} \cite{Ponce:2013}*{Chapter~6}}.
A more precise version of Kato's inequality can be found for instance in \cites{DalMaso_Murat_Orsina_Prignet:1999, Brezis_Ponce:2004}, although Lemma~\ref{lemmaKato} suffices for our purposes in this paper. 

If \(u\) is a supersolution for the Schr\"{o}dinger operator with potential \(V \geq 0\) --- as in the statement of Proposition~\ref{propositionStrongMaximumPrincipleLebesgue} ---, then it follows from Kato's inequality above with \(f = Vu\) that, for every \(\kappa > 0\), we have
\[
\Delta \min{\{u, \kappa\}}
\le \chi_{\{u < \kappa\}} Vu
\le V \min{\{u, \kappa\}}
\] 
in the sense of distributions in \(\Omega\), whence \(\min{\{u, \kappa\}}\) is also a supersolution.
In particular, by Schwartz's characterization of nonnegative distributions~\cite{Schwartz:1945}, \(\Delta \min{\{u, \kappa\}}\) is a locally finite measure, and this implies by interpolation that \(\min{\{u, \kappa\}} \in W\loc^{1, 2}(\Omega)\).
We may thus apply the proposition above with \(\min{\{u, \kappa\}}\), and deduce that \(u = 0\) almost everywhere in \(\Omega\).

The proof of Proposition~\ref{propositionStrongMaximumPrincipleLebesgue} still applies under the weaker assumption that 
\[
\lim_{r \to 0}{\fint\limits_{B(x; r)} u} = 0
\]
in a compact subset with positive \(W^{1, 2}\) capacity.
Indeed, this assumption guarantees that the Poincar\'{e} inequality holds for the function \(\log{(1 + u)}\) and the rest of the proof remains unchanged.
This argument due to Brezis and Ponce~\cite{Brezis_Ponce:2003} provides Ancona's strong maximum principle for potentials \(V \in L^1(\Omega)\) in full generality.

\section{Existence of solutions for the Schr\"{o}dinger operator with measure data}

Another ingredient --- interesting on its own --- in the proof of Theorem~\ref{theoremStrongMaximumPrinciple} concerns the existence of solutions of the Dirichlet problem for the Schrödinger operator with measure data,
\begin{equation}
\label{eqDirichletProblem}
\left\{
\begin{alignedat}{2}
- \Delta v + V v & = \mu	&& \quad \text{in \(\Omega\),}\\
v & = 0 && \quad \text{on \(\partial\Omega\).}
\end{alignedat}
\right.
\end{equation}
We look for solutions of this problem in the sense of Littman, Stampacchia and Weinberger~\cite{Littman_Stampacchia_Weinberger:1963}*{Definition~5.1}.
More precisely, given a finite Borel measure \(\mu\) in \(\Omega\) and \(V \in L^1(\Omega)\), we say that \(v \in L^1(\Omega)\) satisfies the linear Dirichlet problem above if \(V v \in L^1(\Omega)\) and if, for every \(\zeta \in C^\infty(\overline\Omega)\) such that \(\zeta = 0\) on \(\partial\Omega\), we have
\[
\int\limits_\Omega v (-\Delta\zeta + V\zeta)
= \int\limits_\Omega \zeta \dif\mu.
\] 
In the sequel, we denote this class of test functions \(\zeta\) by \(C_0^\infty(\overline\Omega)\).{}
For smooth bounded domains, this notion of solution is equivalent to asking that \(v \in W_0^{1, 1}(\Omega)\) and that the equation is satisfied in the sense of distributions in \(\Omega\) \citelist{\cite{Ponce:2012}*{Corollary~4.5} \cite{Ponce:2013}*{Chapter~6}}.

\begin{proposition}
\label{propositionExistenceDirichletProblem}
Let \(\Omega \subset \R^N\) be a smooth bounded open set, \(p > 1\) and let \(V \in L^p(\Omega)\) be a nonnegative function.
For every nonnegative finite Borel measure \(\mu\) in \(\Omega\) such that \(\mu \in \big( W^{2, p}(\Omega) \cap W_0^{1, p}(\Omega) \big)'\) there exists \(v \in L^{p'}(\Omega)\) satisfying the Dirichlet problem \eqref{eqDirichletProblem}.
\end{proposition}

We denote by \(p'\) the conjugate exponent of \(p\),
\[
\frac{1}{p} + \frac{1}{p'} = 1.
\]
The assumption
\(\mu \in \big( W^{2, p}(\Omega) \cap W_0^{1, p}(\Omega) \big)'\) means that there exists a constant \(C > 0\) such that, for every \(\zeta \in C_0^\infty(\overline\Omega)\), we have
\begin{equation}
\label{eqFunctionalEstimate}
\biggabs{\int\limits_\Omega \zeta \dif\mu}
\le C \norm{\zeta}_{W^{2, p}(\Omega)}.
\end{equation}
By density of \(C_0^\infty(\overline\Omega)\), this is equivalent to the existence of a --- unique --- continuous extension to \(W^{2, p}(\Omega) \cap W_0^{1, p}(\Omega)\) of the linear functional
\[
\zeta \in C_0^\infty(\overline\Omega)
\longmapsto \int\limits_\Omega \zeta \dif\mu.
\]

When \(p > \frac{N}{2}\), the existence of solutions of the Dirichlet problem is proved by Stampacchia~\cite{Stampacchia:1965}*{Th\'{e}orème~9.1}.
In this case, every finite Borel measure \(\mu\) satisfies \(\mu \in \big( W^{2, p}(\Omega) \cap W_0^{1, p}(\Omega) \big)'\) by the Morrey-Sobolev inequality, and the existence of solutions is obtained using the Riesz representation theorem in Lebesgue spaces.

The functional estimate \eqref{eqFunctionalEstimate} is equivalent to the fact that the solution of the Dirichlet problem
\begin{equation}
\label{comparison}
\left\{
\begin{alignedat}{2}
- \Delta w & = \mu	&& \quad \text{in \(\Omega\),}\\
w & = 0 && \quad \text{on \(\partial\Omega\),}
\end{alignedat}
\right.
\end{equation}
belongs to \(L^{p'}(\Omega)\). 
We explain the direct implication, which we shall need in the proof of Proposition~\ref{propositionExistenceDirichletProblem}.
By the assumption on \(\mu\) and by the Calder\'{o}n-Zygmund elliptic estimates~\cite{Gilbarg_Trudinger:1998}*{Theorem~9.14}, for every \(\zeta \in C_0^\infty(\overline\Omega)\) we have
\[
\biggabs{\int\limits_\Omega w \Delta\zeta}
=
\biggabs{\int\limits_\Omega \zeta \dif\mu}
\le 
C \norm{\zeta}_{W^{2, p}(\Omega)}
\le 
C' \norm{\Delta\zeta}_{L^{p}(\Omega)}.
\]
Thus, for every \(\psi \in C^\infty(\overline\Omega)\), we get
\[
\biggabs{\int\limits_\Omega w \psi}
\le 
C' \norm{\psi}_{L^{p}(\Omega)},
\]
and this implies 
\(
w \in L^{p'}(\Omega)
\).{}

\medskip
It is also possible to show that for every compact set \(K \subset \Omega\) with positive \(W^{2, p}\) capacity there exists a positive finite Borel measure \(\mu\) supported in \(K\) such that \(\mu \in \big( W^{2, p}(\Omega) \cap W_0^{1, p}(\Omega) \big)'\).
This is an application of the Hahn-Banach theorem. 
Indeed, the function \(p : C^0(K) \to \R\) defined for all continuous functions \(f : K \to \R\)
by
\[
p(f)
= \inf{\big\{ \norm{\zeta}_{W^{2, p}(\Omega)} : \zeta \in C_0^\infty(\overline\Omega),\ \zeta \ge f\ \text{in \(K\)} \big\}}
\]
is a sublinear function, and \(p(\chi_K) > 0\) by our assumption on the \(W^{2, p}\) capacity of \(K\).
By the Hahn-Banach theorem, there exists a nontrivial linear functional \(L : C^0(K) \to \R\) such that \(L \le p\).
In particular, \(L\) is nonnegative, whence by the Riesz representation theorem in \(C^0(K)\) the functional \(L\) can be written in terms of a positive measure \(\mu\) \cite{Ponce:2013}*{Appendix~A}.

\begin{proof}[Proof of Proposition~\ref{propositionExistenceDirichletProblem}]
We apply an approximation argument based on the potential \(V\).
For this purpose,  let \((V_i)_{i \in \N}\) be a nondecreasing sequence of nonnegative bounded potentials converging pointwisely to \(V\) --- each \(V_i\) could be taken as a truncation of \(V\).
By Stampacchia's existence result for bounded potentials, for each \(i \in \N\) there exists a function \(v_i\) satisfying the Dirichlet problem with potential \(V_i\),
\[
\left\{
\begin{alignedat}{2}
- \Delta v_i + V_i v_i & = \mu	&& \quad \text{in \(\Omega\),}\\
v_i & = 0 && \quad \text{on \(\partial\Omega\).}
\end{alignedat}
\right.
\]
Using Kato's inequality (Lemma~\ref{lemmaKato}), we show that the sequence \((v_i)_{i \in \N}\) is (1) nonnegative and (2) non-increasing.
To verify the first assertion, we observe that since the measure \(\mu\) is nonnegative,
\[
\Delta v_i \le V_i v_i 
\quad \text{in the sense of distributions in \(\Omega\).}
\]
Since the potential \(V_i\) is nonnegative, it follows from Kato's inequality that
\[
\Delta \min{\{v_i, 0\}}
\le \chi_{\{v_i < 0\}} V_i v_i
\le 0
\]
in the sense of distributions in \(\Omega\).
Applying the weak maximum principle (Lemma~\ref{lemmaWeakMaximumPrinciple}), we deduce that \(\min{\{v_i, 0\}} \ge 0\) almost everywhere in \(\Omega\), whence \(v_i\) is nonnegative.

For the second assertion, we subtract the equations satisfied by \(v_i\) and \(v_{i+1}\).
Since \(v_i\) is nonnegative and \(V_{i + 1} \ge V_i\),
\[
\Delta (v_i - v_{i + 1}) 
\le V_i(v_i - v_{i+1})
\quad
\text{in the sense of distributions in \(\Omega\).}
\]
We deduce as above that \(v_i - v_{i+1}\) is nonnegative.

The weak maximum principle (Lemma~\ref{lemmaWeakMaximumPrinciple}) implies that, for every \(i \in \N\), 
\[
v_i \le w,
\]
where \(w\) is the solution of the Dirichlet problem \eqref{comparison}.
It follows from the Monotone convergence theorem that the sequence \((v_i)_{i \in \N}\) converges in \(L^1(\Omega)\) to its pointwise limit \(v\).
By the functional assumption on the measure \(\mu\), we have \(w \in L^{p'}(\Omega)\), whence the nonnegative pointwise limit \(v\) also belongs to \(L^{p'}(\Omega)\).
In addition, 
\[
0 \le V_i v_i \le Vw,
\]
where the function in the right-hand side belongs to \(L^1(\Omega)\).
By the Dominated convergence theorem, we deduce that the sequence \((V_i v_i)_{i \in \N}\) converges in \(L^1(\Omega)\) to \(V v\).
Therefore, the function \(v\) satisfies the Dirichlet problem \eqref{eqDirichletProblem} with potential \(V\).
\end{proof}

There is an alternative proof of Proposition~\ref{propositionExistenceDirichletProblem} based on the method of sub and supersolutions via Schauder's fixed point theorem.
Note that the function identically zero is a subsolution, and \(w\) is a supersolution by the functional assumption on \(\mu\).
We refer the reader to \citelist{\cite{Montenegro_Ponce:2008} \cite{Ponce:2012}*{Proposition~6.7} \cite{Ponce:2013}*{Chapter~19}} for the implementation of this strategy. 

\medskip
The class of measures for which the Dirichlet problem \eqref{eqDirichletProblem} has a solution is actually larger and includes all finite Borel measures \(\mu\) which are diffuse with respect to the \(W^{2, p}\) capacity. 
By diffuse we mean that for every compact set \(K \subset \Omega\) such that \(\capt_{W^{2, p}}{(K)} = 0\), we have \(\mu(K) = 0\).

\begin{corollary}
\label{corollaryExistenceDirichletProblemDiffuseMeasure}
Let \(\Omega \subset \R^N\) be a smooth bounded open set, \(p > 1\) and let \(V \in L^p(\Omega)\) be a nonnegative function.
For every finite Borel measure \(\mu\) which is diffuse with respect to the \(W^{2, p}\) capacity, the Dirichlet problem 
\[
\left\{
\begin{alignedat}{2}
- \Delta v + V v & = \mu	&& \quad \text{in \(\Omega\),}\\
v & = 0 && \quad \text{on \(\partial\Omega\),}
\end{alignedat}
\right.
\]
has a solution.
\end{corollary}

In this case, \(V v \in L^1(\Omega)\) but it need not be true that \(v \in L^{p'}(\Omega)\).
The corollary above has a counterpart for potentials \(V \in L^1(\Omega)\) and for measures which are diffuse with respect to the \(W^{1, 2}\) capacity~\cite{Orsina_Ponce:2008}*{Theorem~1.2}.

We do not use this corollary in the sequel, so we only give a sketch of the proof.
This existence result follows from two main tools.
The first one concerns the absorption estimate,
\begin{equation}
\label{eqAbsorptionEstimate}
\norm{V v}_{L^1(\Omega)}
\le \abs{\mu}(\Omega)
\end{equation}
which can be obtained using as test function a suitable approximation of the sign function \(\sgn{v}\) \citelist{\cite{Brezis_Marcus_Ponce:2007}*{Proposition~4.B.3} \cite{Ponce:2012}*{Lemma~7.2} \cite{Ponce:2013}*{Chapter~20}}.
The second ingredient is a property of strong approximation of nonnegative measures which are diffuse with respect to the \(W^{2, p}\) capacity by nonnegative measures in \(\big( W^{2, p}(\Omega) \cap W_0^{1, p}(\Omega) \big)'\) \citelist{\cite{Feyel_DelaPradelle:1977} \cite{DalMaso:1983} \cite{Baras_Pierre:1984} \cite{Gallouet_Morel:1984} \cite{Boccardo_Gallouet_Orsina:1996}};
we refer the reader to \citelist{ \cite{Ponce:2012}*{Proposition~7.6} \cite{Ponce:2013}} for the complete argument.

\begin{proof}[Proof of Corollary~\ref{corollaryExistenceDirichletProblemDiffuseMeasure}]
Since the equation is linear and the measure \(\mu\) can be written as a difference of nonnegative diffuse measures --- for instance the positive and negative parts of \(\mu\) ---, we may assume without loss of generality that \(\mu\) is nonnegative.
By the property of strong approximation of diffuse measures, there exists a sequence \((\mu_i)_{i \in \N}\) in \(\big( W^{2, p}(\Omega) \cap W_0^{1, p}(\Omega) \big)'\) such that
\[
\lim_{i \to \infty}{\abs{\mu - \mu_i}(\Omega)} = 0.
\]
By Proposition~\ref{propositionExistenceDirichletProblem}, the Dirichlet problem with datum \(\mu_i\) has a solution \(v_i\).
By the absorption estimate \eqref{eqAbsorptionEstimate} and the strong convergence of the sequence of measures \((\mu_i)_{i \in \N}\), we deduce that \((V v_i)_{i \in \N}\) is a Cauchy sequence in \(L^1(\Omega)\).
Thus, the sequence of measures \((\Delta v_i)_{i \in \N}\) converges strongly in the sense of measures, whence  \((v_i)_{i \in \N}\) is a Cauchy sequence in \(L^1(\Omega)\) and converges strongly to a function \(v\).
In particular, the sequence \((V v_i)_{i \in \N}\) converges in \(L^1(\Omega)\) to the function \(V v\).
Therefore, \(v\) satisfies the Dirichlet problem with datum \(\mu\).
\end{proof}

\section{Choice of test functions}
\label{sectionChoiceTestFunctions}

In this section we explain how we can enlarge the class of nonnegative test functions used in the differential inequality \eqref{eqPDEInequality}: from \(C_c^\infty(\Omega)\) functions to solutions of a Dirichlet problem with measure data.
The first step consists in passing from test functions with compact support in \(\Omega\) to test functions merely vanishing on the boundary \(\partial\Omega\).
The main ingredient is the following:

\begin{proposition}
\label{propositionTestFunction}
Let \(\Omega \subset \R^N\) be a smooth bounded open set and let \(w \in W_0^{1, 1}(\Omega)\) be a function such that \(\Delta w\) is a finite Borel measure in \(\Omega\).
If \(w\) is nonnegative, then for every nonnegative function \(\psi \in C^\infty(\overline\Omega)\) we have
\[
\int\limits_\Omega \psi \Delta w
\le \int\limits_\Omega w \Delta \psi
.
\]
\end{proposition}

The integral in the left-hand side is to be understood as the integration of \(w\) with respect to the measure \(\Delta w\); we avoid the notation \(\dif(\Delta w)\).
In the proof of Theorem~\ref{theoremStrongMaximumPrinciple}, we choose as \(\psi\) a regularized version of \(u\) via convolution.

Observe that if \(w \in C_0^\infty(\overline\Omega)\), then by the Divergence theorem we have, for every \(\psi \in C^\infty(\overline\Omega)\),
\[
\int\limits_\Omega w \Delta \psi
= \int\limits_\Omega \psi \Delta w - \int\limits_{\partial\Omega} \frac{\partial w}{\partial n} \psi,
\]
where \(n\) denotes the exterior normal derivative on \(\partial\Omega\).
When \(w\) and \(\psi\) are both nonnegative, the integrand on the boundary \(\partial\Omega\) is nonpositive and we get the inequality.
For \(w\) as in the statement of the proposition, we rigorously justify this argument by studying an extension of \(w\) to \(\R^N\).

\begin{proof}[Proof of Proposition~\ref{propositionTestFunction}]
Consider the extension \(\overline w : \R^N \to \R\) defined by
\[
\overline w(x) =
\begin{cases}
w(x)	& \text{if \(x \in \Omega\),}\\
0		& \text{if \(x \in \R^N \setminus \Omega\).}
\end{cases}
\]
Since \(w \in W_0^{1, 1}(\Omega)\) and \(\Delta w\) is a finite Borel measure in \(\Omega\), one shows that \citelist{\cite{Brezis_Ponce:2008}*{Proposition~4.2} \cite{Ponce:2013}*{Chapter~10}} (1) \(\Delta \overline w\) is a finite Borel measure in \(\R^N\) supported in \(\overline\Omega\), and (2) there exists a measure \(\nu\) supported in \(\partial\Omega\) such that, for every Borel set \(A \subset \R^N\), we have
\[
\Delta\overline w(A)
= \Delta w(A \cap \Omega) + \nu(A \cap \partial\Omega).
\]
Hence, using any smooth extension \(\widetilde\psi\) of \(\psi\) with compact support in \(\R^N\), we get
\[
\int\limits_\Omega w \Delta \psi
= \int\limits_{\R^N} \overline w \Delta \widetilde\psi
= \int\limits_{\R^N} \widetilde\psi \Delta \overline w
= \int\limits_{\Omega} \psi \Delta w + \int\limits_{\partial\Omega} \psi \dif\nu.
\]
To conclude, we need a property discovered by de~la~Vall\'ee-Poussin~\cite{DLVP:1938} and generalized by Brelot~\cite{Brelot:1951}.
It says that the diffuse part of the measure \(\Delta \overline w\) with respect to the \(W^{1, 2}\) capacity is nonnegative on the minimum set of the precise representative of \(\overline w\) \citelist{\cite{Brezis_Ponce:2004}*{Corollary~1.3} \cite{Ponce:2013}*{Chapter~6}}.
In our case, the measure \(\nu\) is absolutely continuous with respect to the Haudorff measure \(\cH^{N-1}\lfloor_{\partial\Omega}\) \citelist{\cite{Ancona:2009} \cite{Brezis_Ponce:2008}*{Proposition~4.2} \cite{Ponce:2013}*{Chapter~10}};
in particular \(\nu\) is diffuse with respect to the \(W^{1, 2}\) capacity.
Since \(w\) is nonnegative and has zero trace on \(\partial\Omega\), \(\nu\) is supported in the set where \(\overline w\) achieves its minimum, whence by the de~la~Vallée Poussin property \(\nu\) is nonnegative and
\[
\int\limits_{\partial\Omega} \psi \dif\nu \ge 0
.
\]
The conclusion follows.
\end{proof}

The second step consists in constructing \emph{nonnegative} solutions \(w\) of a Dirichlet problem involving the Schrödinger operator \(-\Delta + V\) in such a way that \(-\Delta w + Vw\) is nonnegative in a prescribed set; in the context of Theorem~\ref{theoremStrongMaximumPrinciple}, a subset where \(u\) vanishes.

\begin{proposition}
\label{propositionComparisonSolutions}
Let \(\Omega \subset \R^N\) be a smooth bounded open set and let \(V \in L^1(\Omega)\) be a nonnegative function.
If \(\mu\) is a positive finite Borel measure in \(\Omega\) such that there exists a function \(v\) satisfying the Dirichlet problem
\[
\left\{
\begin{alignedat}{2}
- \Delta v + V v & = \mu	&& \quad \text{in \(\Omega\),}\\
v & = 0 && \quad \text{on \(\partial\Omega\),}
\end{alignedat}
\right.
\]
then there exists \(C > 0\) such that for every \(\epsilon > 0\) the solution \(v_\epsilon\) of the Dirichlet problem
\[
\left\{
\begin{alignedat}{2}
- \Delta v_\epsilon + V v_\epsilon & = \chi_{\{v > \epsilon\}}	&& \quad \text{in \(\Omega\),}\\
v_\epsilon & = 0 && \quad \text{on \(\partial\Omega\),}
\end{alignedat}
\right.
\]
satisfies
\(\epsilon v_\epsilon \le C v\) almost everywhere in \(\Omega\).
\end{proposition}

The existence of \(v_{\epsilon}\), for every \(\epsilon > 0\), is obtained for example by minimization of the functional
\[{}
E(u)
= \frac{1}{2} \int\limits_{\Omega} (\abs{\nabla u}^{2} + Vu^{2}) - \int\limits_{\Omega} f u
\]
in \(W_{0}^{1, 2}(\Omega)\) with bounded function \(f = \chi_{\{v > \epsilon\}}\).
In the proof of this proposition we need the following weak maximum principle adapted to solutions of the Dirichlet problem in the weak sense~\citelist{\cite{Brezis_Marcus_Ponce:2007}*{Proposition~4.B.1} \cite{Ponce:2012}*{Corollary~4.5 and Proposition~5.1} \cite{Ponce:2013}*{Chapter~6}}:

\begin{lemma}
\label{lemmaWeakMaximumPrinciple}
Let \(\Omega \subset \R^N\) be a smooth bounded open set.
If \(v \in W_0^{1, 1}(\Omega)\) is such that
\[
\Delta v \le 0
\quad \text{in the sense of distributions in \(\Omega\),}
\]
then \(v \ge 0\) almost everywhere in \(\Omega\).
\end{lemma}

The proof of this lemma is based  on an approximation of functions in \(C_0^\infty(\overline\Omega)\) by functions in \(C_c^\infty(\Omega)\).
One deduces that for every nonnegative function \(\zeta \in C_0^\infty(\overline\Omega)\),
\[
- \int\limits_\Omega v \Delta\zeta
= \int\limits_\Omega \nabla v \cdot \nabla \zeta \ge 0,
\]
which implies that \(v\) is nonnegative.

\begin{proof}[Proof of Proposition~\ref{propositionComparisonSolutions}]
We first observe that the family \((v_\epsilon)_{\epsilon > 0}\) is uniformly bounded. 
More precisely, we show that for every \(\epsilon > 0\) we have
\[
v_\epsilon \le \zeta
\quad \text{in \(\Omega\)},
\]
where \(\zeta\) is the solution of the Dirichlet problem
\[
\left\{
\begin{alignedat}{2}
- \Delta \zeta & = 1	&& \quad \text{in \(\Omega\),}\\
\zeta & = 0 && \quad \text{on \(\partial\Omega\).}
\end{alignedat}
\right.
\]
Note that \(\zeta\) is a supersolution of the equation satisfied by \(v_\epsilon\) since
\[
-\Delta \zeta + V \zeta \ge -\Delta \zeta = 1 \ge \chi_{\{v > \epsilon\}}
\]
in the sense of distributions in \(\Omega\).
Then, by Kato's inequality (Lemma~\ref{lemmaKato}), we have
\[
\Delta \min{\{\zeta - v_\epsilon, 0\}} 
\le \chi_{\{\zeta < v_\epsilon\}} V(\zeta - v_\epsilon).
\]
By nonnegativity of \(V\), we deduce that
\[
\Delta \min{\{\zeta - v_\epsilon, 0\}} 
\le 0
\quad \text{in the sense of distributions in \(\Omega\).}
\]
The weak maximum principle (Lemma~\ref{lemmaWeakMaximumPrinciple}) gives \(\min{\{\zeta - v_\epsilon, 0\}} \ge 0\) almost everywhere in \(\Omega\), whence \(v_\epsilon \le \zeta
\).

\smallskip
We claim that 
\[
\epsilon v_\epsilon \le C v \quad \text{in \(\Omega\),}
\]
where the constant \(C > 0\) is such that, for every \(x \in \overline\Omega\),
\[
\zeta(x) \le C.
\] 
Firstly, since
\[
\Delta(C v - \epsilon v_\epsilon)
\le V(Cv - \epsilon v_\epsilon) + \epsilon \chi_{\{v > \epsilon\}}
\]
we have, by Kato's inequality (Lemma~\ref{lemmaKato}) and  by nonnegativity of \(V\),
\[
\begin{split}
\Delta \min{\big\{C v - \epsilon v_\epsilon, 0 \big\}}
& \le \chi_{\{Cv < \epsilon v_\epsilon\}} \big[ V(Cv - \epsilon v_\epsilon) + \epsilon \chi_{\{v > \epsilon\}}  \big]\\
& \le \epsilon \chi_{\{Cv < \epsilon v_\epsilon\}}  \chi_{\{v > \epsilon\}}
\end{split}
\]
in the sense of distributions in \(\Omega\).
By the choice of the constant \(C\), for every \(x \in \Omega\) such that \(v(x) > \epsilon\) we have
\[
\epsilon v_\epsilon(x) 
\le \epsilon \zeta(x) 
\le C \epsilon
\le C v(x).
\]
Hence,
\[
\{Cv < \epsilon v_\epsilon\} \cap \{v > \epsilon\} = \emptyset.
\]
Thus,
\[
\Delta \min{\big\{C v - \epsilon v_\epsilon, 0 \big\}}
\le 0
\quad \text{in the sense of distributions in \(\Omega\).}
\]
From the weak maximum principle (Lemma~\ref{lemmaWeakMaximumPrinciple}) we deduce that 
\[
\min{\big\{C v - \epsilon v_\epsilon, 0 \big\}} \ge 0
\]
and the proposition follows. 
\end{proof}

\section{Proof of Theorem~\ref{theoremStrongMaximumPrinciple}}

Let \(\omega \Subset \Omega\) be a smooth open connected set containing a compact subset \(K \subset \Omega\) with positive \(W^{2, p}\) capacity such that, for every \(x \in K\),
\[
\lim_{r \to 0}{\fint\limits_{B(x; r)} u} = 0.
\]
By the Hahn-Banach theorem, there exists a positive finite Borel measure \(\mu\) supported in \(K\) such that \(\mu \in \big( W^{2, p}(\omega) \cap W_0^{1, p}(\omega) \big)'\).
Let \(C > 0\) be a constant given by Proposition~\ref{propositionComparisonSolutions} such that for every \(\epsilon > 0\),
\[
\epsilon v_\epsilon \le C v
\quad \text{almost everywhere in \(\omega\)},
\]
where \(v\) and \(v_\epsilon\) are the solutions of the Dirichlet problem in the statement of the proposition with \(\Omega\) replaced by \(\omega\).
The assumption \(V \in L^p(\Omega)\) guarantees the existence of \(v\) and \(v_\epsilon\)  in \(L^{p'}(\omega)\) in view of Proposition~\ref{propositionExistenceDirichletProblem}.

Given a sequence of positive numbers \((\kappa_i)_{i \in \N}\) converging to zero and given a nonnegative function \(\rho \in C_c^\infty(\R^N)\) such that \(\int_{\R^N} \rho = 1\), let \((\rho_i)_{i \in \N}\) be the sequence of mollifiers defined by
\[
\rho_i(x)
= \frac{1}{\kappa_i^N} \rho\big( \tfrac{x}{\kappa_i} \big).
\] 
If \(\kappa_i\) is sufficiently small, then we have \(\diam{(\supp{\rho_i})} \le d(\omega, \partial\Omega)\).
In this case,
\[
\Delta (\rho_i * u)
= \rho_i * \Delta u
\] 
pointwisely in \(\omega\).
Since the function \(\rho_i * u \in C^\infty(\overline\omega)\) is nonnegative and \(C v - \epsilon v_\epsilon\) is also a nonnegative function in \(W_0^{1, 1}(\omega)\) such that 
\(\Delta (C v - \epsilon v_\epsilon)\) is a finite Borel measure in \(\omega\), by Proposition~\ref{propositionTestFunction} we have
\begin{equation}
\label{eqInequalityIntegral}
\int\limits_{\omega} (\rho_i * u) \Delta (Cv - \epsilon v_\epsilon)
\le
\int\limits_{\omega} (Cv - \epsilon v_\epsilon) \Delta (\rho_i * u).
\end{equation}
We now study the limits of the left and right-hand sides as \(i\) tends to infinity.
For this purpose, we first consider the case where \(u\) is a bounded function,
\[
u \in L^\infty(\Omega).
\]
By the differential inequality satisfied by \(u\),
\[
\Delta (\rho_i * u)
= \rho_i * \Delta u
\le \rho_i * (Vu).
\] 
We are assuming that \(u \in L^\infty(\Omega)\), whence the sequence \((\rho_i * (Vu))_{i \in \N}\) converges to \(Vu\) in \(L^p(\omega)\).
Since \(Cv - \epsilon v_\epsilon \in L^{p'}(\Omega)\), we then have
\begin{equation}
\label{eqLimitRHS}
\begin{split}
\limsup_{i \to \infty}{\int\limits_{\omega} (Cv - \epsilon v_\epsilon) \Delta (\rho_i * u) }
& \le
\lim_{i \to \infty}{\int\limits_{\omega} (Cv - \epsilon v_\epsilon) \rho_i * (Vu)}\\
& = \int\limits_{\omega} (Cv - \epsilon v_\epsilon) V u.
\end{split}
\end{equation}
On the other hand, by the equations satisfied by \(v\) and \(v_\epsilon\) we have
\[
\int\limits_{\omega} (\rho_i * u) \Delta (Cv - \epsilon v_\epsilon)
= \int\limits_{\omega} (\rho_i * u) \big[V(Cv - \epsilon v_\epsilon) + \epsilon \chi_{\{v > \epsilon\}} \big] - C \int\limits_{\omega} (\rho_i * u) \dif \mu.
\]
Since \(u \in L^\infty(\Omega)\),
\[
\lim_{i \to \infty}{\int\limits_{\omega} (\rho_i * u) \big[V(Cv - \epsilon v_\epsilon) + \epsilon \chi_{\{v > \epsilon\}} \big]}
= \int\limits_{\omega} u \big[V(Cv - \epsilon v_\epsilon) + \epsilon \chi_{\{v > \epsilon\}} \big].
\]
By assumption, the average integral of \(u\) on balls converges pointwisely to zero in the support of \(\mu\), whence the same is true for the sequence of convolutions \((\rho_i * u)_{i \in \N}\).
Since we are assuming that \(u \in L^\infty(\Omega)\), by the Dominated convergence theorem we have
\[
\lim_{i \to \infty}{\int\limits_{\omega} (\rho_i * u) \dif \mu} = 0.
\] 
Hence,
\begin{equation}
\label{eqLimitLHS}
\lim_{i \to \infty}{\int\limits_{\omega} (\rho_i * u) \Delta (Cv - \epsilon v_\epsilon)
}
= \int\limits_{\omega} u \big[V(Cv - \epsilon v_\epsilon) + \epsilon \chi_{\{v > \epsilon\}} \big].
\end{equation}
Therefore, as \(i\) tends to infinity in \eqref{eqInequalityIntegral}, it follows from the limits \eqref{eqLimitRHS} and \eqref{eqLimitLHS} that
\[
\int\limits_{\omega} u \big[V(Cv - \epsilon v_\epsilon) + \epsilon \chi_{\{v > \epsilon\}} \big]
\le \int\limits_{\omega} (Cv - \epsilon v_\epsilon) V u.
\]
Simplifying the common term on both sides,  we get
\[
\epsilon \int\limits_{\omega} u \chi_{\{v > \epsilon\}} 
\le
0.
\]
Thus, dividing both sides by \(\epsilon\) and letting \(\epsilon\) tend to zero, we get
\[
\int\limits_{\{v > 0\}} u 
\le
0.
\]
Since by the strong maximum principle involving the Lebesgue measure (Proposition~\ref{propositionStrongMaximumPrincipleLebesgue})  the set \(\{v = 0\}\) is negligible, and since \(u\) is nonnegative, we deduce that \(u = 0\) almost everywhere in \(\omega\).
Since the domain \(\Omega\) can be covered by the sets \(\omega\), we get the conclusion when \(u \in L^\infty(\Omega)\).

We may now remove this restriction on \(u\) by observing that, by Kato's inequality (Lemma~\ref{lemmaKato}), for every \(\kappa > 0\) the function \(\min{\{u, \kappa\}}\) satisfies the same differential inequality as \(u\):
\[
- \Delta \min{\{u, \kappa\}} + V \min{\{u, \kappa\}} \ge 0
\]
in the sense of distributions in \(\Omega\).
Moreover, since \(0 \le \min{\{u, \kappa\}} \le u \), the assumption on the limit of the average integral of \(\min{\{u, \kappa\}}\) is satisfied.
By the previous case, 
\[
\min{\{u, \kappa\}} = 0 \quad \text{almost everywhere in \(\Omega\),}
\]
whence \(u = 0\) almost everywhere in \(\Omega\).
The proof of the theorem is complete.
\qed

\section{Prescribing the level set $\{u = 0\}$}
\label{sectionConverse}

In this section, we investigate the role played by the \(W^{2, p}\) capacity in the strong maximum principle by proving the following converse of Theorem~\ref{theoremStrongMaximumPrinciple}. 
Later on, we consider the counterpart of the case \(p = 1\) in terms of the \(W^{1, 2}\) capacity.

\begin{proposition}
\label{propositionExample}
Let \(\Omega \subset \R^N\) be an open set and \(p > 1\).
For every compact set \(K \subset \Omega\) with zero \(W^{2, p}\) capacity there exist a nonnegative function \(u \in C^\infty(\overline\Omega)\) and  \(V \in L^p(\Omega)\) such that
\[
K = \{x \in \overline{\Omega} : u(x) = 0\}, 
\]
and the equation
\[
- \Delta u + V u = 0	
\]
is satisfied pointwisely and in the sense of distributions in \(\Omega\).
\end{proposition}

The idea is to construct a nonnegative function \(u\) of the form \(\frac{1}{w}\) where \(w \in C^\infty(\overline\Omega \setminus K)\) is a function diverging to \(+\infty\) in \(K\).
In this case, we have pointwisely in \(\Omega \setminus K\) the identity
\[
\Delta \Big(\frac{1}{w}\Big)
=
\bigg( {- {\frac{\Delta w}{w}}} + 2 \frac{\abs{\nabla w}^2}{w^2}  \bigg) \frac{1}{w}.
\]

The heart of the matter is to find a suitable estimate for the function in parentheses in the right-hand side.
For this purpose we need the following estimate:

\begin{lemma}
\label{lemmaMazya}
For every \(p \ge 1\) and for every nonnegative function \(\varphi \in C_c^\infty(\R^N)\), we have
\[
\int\limits_{\R^N} \frac{\abs{\nabla\varphi}^{2p}}{(1 + \varphi)^{2p}}
\le C \int\limits_{\R^N} \frac{\abs{D^2\varphi}^p}{(1 + \varphi)^{p}},
\]
for some constant \(C > 0\) depending on \(p\).
\end{lemma}

\begin{proof}
We rely on the pointwise identity
\[
\div{\bigg[ \frac{\abs{\nabla\varphi}^{2p - 2} \nabla\varphi}{(1 + \varphi)^{2p-1}}  \bigg]}
= - (2p-1) \frac{\abs{\nabla\varphi}^{2p}}{(1 + \varphi)^{2p}}
+ \frac{\div{\big(\abs{\nabla\varphi}^{2p - 2}}  \nabla\varphi \big)}{(1 + \varphi)^{2p-1}}.
\]
Applying the Divergence theorem, we have
\[
\int\limits_{\R^N} \frac{\abs{\nabla\varphi}^{2p}}{(1 + \varphi)^{2p}}
\le C \int\limits_{\R^N} \frac{\abs{D^2\varphi} \abs{\nabla\varphi}^{2(p-1)}}{(1 + \varphi)^{2p-1}}.
\]
This is the estimate we want when \(p = 1\).{}
In the case \(p > 1\), we obtain the conclusion applying H\"{o}lder's inequality in the right-hand side.
\end{proof}

The lemma above is reminiscent of Maz'ya's inequality \cite{Mazya:1973}*{proof of Theorem~11} valid for \(p > 1\):
\[
\int\limits_{\R^N} \frac{\abs{\nabla\varphi}^{2p}}{(1 + \varphi)^p}
\le C \int\limits_{\R^N} \abs{D^2\varphi}^p,
\]
with the same proof.

Before proving the proposition, we also observe that for any compact set \(K \subset \R^{N}\) with zero \(W^{2, p}\) capacity, we may choose in the definition of the capacity of \(K\) a minimizing sequence \((\varphi_i)_{i \in \N}\) in \(C_c^\infty(\R^{N})\) with support in some fixed open set \(\omega \supset K\).
Indeed, it suffices to multiply any given minimizing sequence in \(C_c^\infty(\R^{N})\) by some fixed nonnegative function  in \(C_c^\infty(\omega)\) which is greater than or equal to \(1\) in \(K\).
Thus, for every \(\epsilon > 0\) and for every open set \(\omega \supset K\), there exists a nonnegative function \(\varphi \in C_c^\infty(\omega)\) such that 
\[
\norm{\varphi}_{W^{2, p}(\R^{N})} \le \epsilon
\]
and \(\varphi > 1\) in \(K\).

\begin{proof}[Proof of Proposition~\ref{propositionExample}]
Let \((\omega_i)_{i \in \N}\) be a non-increasing sequence of open subsets of \(\Omega\) containing \(K\) such that
\[
\bigcap_{i \in \N}{\omega_i} = K.
\]
Given a sequence of positive numbers \((\epsilon_i)_{i \in \N}\), we construct by induction a sequence of nonnegative functions \((\varphi_i)_{i \in \N}\) in \(C_c^\infty(\Omega)\) such that, for every \(i \in \N\), we have  
\begin{enumerate}[(a)]
\item 
\label{itemA} 
\(
\norm{\varphi_i}_{W^{2, p}(\Omega)} \le \epsilon_i,
\) 
\item 
\label{itemB} 
\(\varphi_i > 1\) in \(K\),
\item
\label{itemC} 
\(
\supp{\varphi_{i+1}} \subset \omega_i \cap \{\varphi_{i} > 1\}.
\)
\end{enumerate}

We now consider the sequence of functions \((w_j)_{j \in \N}\) defined by
\begin{equation}
	\label{eqChoice}
w_j = 1 + \sum_{i=0}^j \alpha_{i} \varphi_i,
\end{equation}
where \((\alpha_{i})_{i \in \N}\) is a sequence of real numbers such that \(\alpha_{i} \ge 1\) for every \(i \in \N\).{}
The explicit choice of \((\alpha_{i})_{i \in \N}\) will ensure the smoothness of the pointwise limit of the sequence \((\frac{1}{w_{j}})_{j \in \N}\).

By property~\eqref{itemC}, for every \(k, l \in \N\) such that \(k \ge l\) we have 
\begin{equation}
	\label{eqStationary}
	w_{k} = w_{l} 
	\quad \text{in \(\overline\Omega \setminus \omega_{l}\).}
\end{equation}
Thus, the sequence \((w_{j})_{j \in \N}\) is stationary and, in particular, converges in \(\overline\Omega \setminus K\).
On the other hand, if \(x \in K\), then by property~\eqref{itemB} we have \(w_{j}(x) \ge j + 1\) for every \(j \in \N\).
Therefore, \(K\) is the set where the sequence \((w_{j})_{j \in \N}\) diverges pointwisely to \(+\infty\).

For every \(j \in \N\),  we have \(w_j \in C^\infty(\overline\Omega)\) and
\begin{equation}
	\label{eqIdentityPDE}
	\Delta \Big(\frac{1}{w_j}\Big)
	= \bigg( {- {\frac{\Delta w_j}{w_j}}} + 2 \frac{\abs{\nabla w_j}^2}{w_j^2}  \bigg) \frac{1}{w_j}.
\end{equation}
The sequence \((\frac{1}{w_j})_{j \in \N}\) converges uniformly in \(\overline\Omega\).
Indeed, by property~\eqref{itemC} for every \(k, l \in \N\) such that \(k \ge l\) we have \(w_k = w_l\) in \(\overline\Omega \setminus \{\varphi_l > 1\}\).{}
Since \(w_k \ge w_l \ge l+1\) in \(\{\varphi_l > 1\}\), we get
\[
\Bignorm{\frac{1}{w_k} - \frac{1}{w_l}}_{L^\infty(\Omega)}
= \Bignorm{\frac{1}{w_k} - \frac{1}{w_l}}_{L^\infty(\{\varphi_l > 1\})}
\le \frac{1}{l+1}.
\]
By \eqref{eqStationary}, the sequence of functions \((V_j)_{j \in \N}\) defined by
\[
V_j = 
- \frac{\Delta w_j}{w_j} + 2 \frac{\abs{\nabla w_j}^2}{w_j^2}
\]
is also pointwisely stationary in \(\Omega \setminus K\), and we take a measurable function  \(V : \Omega \to \R\) such that, for every \(x \in \Omega \setminus K\),{}
\[{}
V(x)
= \lim_{j \to \infty}{V_{j}(x)}.
\]

\begin{claim}
	For every \(j \in \N\), we have
\[
\norm{V_j}_{L^p(\Omega)}
\le C' \bigg[\sum_{i=0}^j \epsilon_i + \Big(\sum_{i=0}^j \epsilon_i^{1/2} \Big)^2 \bigg].
\]
\end{claim}

\begin{proof}[Proof of the claim]
By the triangle inequality and by the inequality \(w_{j} \ge 1\), we have
\begin{equation}
\label{eqEstimate1}
\Bignorm{\frac{\Delta w_j}{w_j}}_{L^p(\Omega)}
\le \sum_{i=0}^j \Bignorm{\frac{\Delta \varphi_i}{w_j}}_{L^p(\Omega)}
\le \sum_{i=0}^j \norm{\Delta\varphi_i}_{L^p(\Omega)}.
\end{equation}
Concerning the second term, by the triangle inequality we have
\[
\Bignorm{\frac{\abs{\nabla w_j}^2}{w_j^2}}_{L^p(\Omega)}
=
\Bignorm{\frac{\nabla w_j}{w_j}}_{L^{2p}(\Omega)}^2
\le \bigg(\sum_{i=0}^j \Bignorm{\frac{{\nabla \varphi_i}}{w_j}}_{L^{2p}(\Omega)} \bigg)^2.
\]
Since for every \(i \le j\) we have \(w_j \ge 1 + \varphi_i\), we may estimate the quantity inside the summation as
\[
\Bignorm{\frac{{\nabla \varphi_i}}{w_j}}_{L^{2p}(\Omega)}
\le \Bignorm{\frac{{\nabla \varphi_i}}{1 + \varphi_i}}_{L^{2p}(\Omega)}.
\]
By the variant of Maz'ya's inequality (Lemma~\ref{lemmaMazya}), we have
\[
\Bignorm{\frac{\nabla \varphi_i}{1 + \varphi_i}}_{L^{2p}(\Omega)}^2
\le C \Bignorm{\frac{D^2 \varphi_i}{1 + \varphi_i}}_{L^p(\Omega)}
\le C \norm{D^2 \varphi_i}_{L^p(\Omega)}.
\]
Therefore,
\begin{equation}
\label{eqEstimate2}
\Bignorm{\frac{\abs{\nabla w_j}^2}{w_j^2}}_{L^p(\Omega)}
\le C\bigg(\sum_{i=0}^j \norm{D^2 \varphi_i}_{L^p(\Omega)}^{1/2} \bigg)^2.
\end{equation}
Combining estimates \eqref{eqEstimate1} and \eqref{eqEstimate2} with property~\eqref{itemA}, the estimate follows.
\end{proof}

Choosing the sequence \((\epsilon_i)_{i \in \N}\) such that the series 
\(
\sum\limits_{i=0}^\infty \epsilon_i^{1/2}
\)
converges, it follows that the sequence \((V_j)_{j \in \N}\) is bounded in \(L^p(\Omega)\).
By Fatou's lemma we deduce that \(V \in L^{p}(\Omega)\), and by Hölder's inequality the sequence \((V_j)_{j \in \N}\) is equi-integrable in \(\Omega\).
Letting \(j\) tend to infinity in the equation \eqref{eqIdentityPDE}, it follows from Vitali's convergence theorem that the uniform limit \(u\) of the sequence \((\frac{1}{w_j})_{j \in \N}\) satisfies
\[
\Delta u = V u \quad \text{in the sense of distributions in \(\Omega\),}
\]
regardless of the choice of the sequence \((\alpha_{i})_{i \in \N}\).{}

We now choose the sequence \((\alpha_{i})_{j \in \N}\) by induction as follows.
Let \(\alpha_{0} = 1\).{}
Take \(\alpha_{0}, \dots, \alpha_{j-1}\) for some \(j \in \N_{*}\), and define \(w_{j - 1}\) accordingly as in \eqref{eqChoice}.
We observe that, for every \(\ell \in \N_{*}\), we have
\begin{equation}
	\label{eqLimitDerivatives}
\lim_{\alpha \to \infty}{\Bignorm{D^{\ell} \Big(\frac{1}{w_{j-1} + \alpha\varphi_{j} + \beta\varphi_{j+1}}\Big)}_{L^{\infty}(\{\varphi_{j} > 1\})}} = 0,
\end{equation}
uniformly with respect to \(\beta \ge 0\).{}
Indeed, by differentiation of composite functions, this uniform limit is a consequence of the one dimensional identity: for every \(k \in \N_{*}\) and for every \(t > 0\),
\[{}
\biggabs{t^{k} \frac{\dif^{k}}{\dif t^{k}}\Big( \frac{1}{t} \Big)} 
= \frac{k!}{t}.
\]
By \eqref{eqLimitDerivatives}, we may take \(\alpha_{j} \ge 1\) such that, for every \(\ell \in \{1, \dots, j\}\) and for every \(\beta \ge 0\), we have
\[{}
\Bignorm{D^{\ell} \Big(\frac{1}{w_{j-1} + \alpha_{j}\varphi_{j} + \beta\varphi_{j+1}}\Big)}_{L^{\infty}(\{\varphi_{j} > 1\})}
\le 1.
\]
This concludes the choice of the sequence \((\alpha_{i})_{i \in \N}\). 
Since 
\[{}
w_{j + 1} = w_{j-1} + \alpha_{j}\varphi_{j} + \alpha_{j+1}\varphi_{j+1},
\]
for every \(j \ge \ell\) we then have
\begin{equation}
 \label{eqEstimateDerivatives}
\Bignorm{D^{\ell} \Big(\frac{1}{w_{j+1}}\Big)}_{L^{\infty}(\{\varphi_{j} > 1\})}
\le 1.
\end{equation}

\begin{claim}
For every \(\ell \in \N_{*}\), the sequence \((D^{\ell} \frac{1}{w_{j}})_{j \in \N}\) is uniformly bounded in \(\overline\Omega\).{}
\end{claim}

\begin{proof}[Proof of the claim]
Given \(j \in \N\) such that \(j \ge \ell\), we decompose the domain as
\[{}
\Omega{}
= \big(\Omega\setminus \{\varphi_{\ell} \le 1\}\big) 
\cup{}
\bigcup_{i = \ell}^{j-1}{\big( \{\varphi_{i} > 1\} \setminus \{\varphi_{i+1} \le 1\} \big)}
\cup{}
\{\varphi_{j} > 1\}.
\]
By property~\eqref{itemC}, we have 
\[{}
w_{j+1} = w_{\ell}
\quad \text{in \(\Omega\setminus \{\varphi_{\ell} \le 1\}\),}
\]
and for every \(i \in \{\ell, \dots, j-1\}\) we also have 
\[{}
w_{j+1} = w_{i+1}
\quad \text{in \(\{\varphi_{i} > 1\} \setminus \{\varphi_{i+1} \le 1\}\).}
\]
Therefore, by estimate \eqref{eqEstimateDerivatives} we obtain 
\[{}
\Bignorm{D^{\ell} \Big(\frac{1}{w_{j+1}}\Big)}_{L^{\infty}(\Omega)}
\le \max{\left\{  \Bignorm{D^{\ell} \Big(\frac{1}{w_{\ell}}\Big)}_{L^{\infty}(\Omega\setminus \{\varphi_{\ell} \le 1\})} , 1 \right\}}.
\]
The right-hand side being independent of \(j \ge \ell\), the sequence \((D^{\ell} \frac{1}{w_{j}})_{j \in \N}\) is thus uniformly bounded in \(\overline\Omega\).
\end{proof}

Since \(w_{j} = 1\) in \(\overline\Omega \setminus \supp{\varphi_{0}}\),
it follows from the claim that the uniform limit \(u\) of the sequence \((\frac{1}{w_{j}})_{j \in \N}\) belongs to \(C^{\infty}(\overline\Omega)\), and for every \(\ell \in \N_{*}\) the sequence \((D^{\ell} \frac{1}{w_{j}})_{j \in \N}\) converges uniformly to \(D^{\ell} u\) in \(\overline\Omega\).
In particular, the sequence \((\Delta \frac{1}{w_{j}})_{n \in \N}\) converges uniformly to \(\Delta u\) in \(\overline{\Omega}\), whence as \(j\) tends to infinity in \eqref{eqIdentityPDE} we get 
\[{}
\Delta u = V u
\quad{}
\text{pointwisely in \(\Omega\).}
\]
This concludes the proof of the proposition.
\end{proof}

\medskip

The previous construction has the following counterpart for \(p = 1\):

\begin{proposition}
\label{propositionExampleW12}
Let \(\Omega \subset \R^N\) be an open set.
For every compact set \(K \subset \Omega\) with zero \(W^{1, 2}\) capacity there exist a nonnegative function \(u \in C^\infty(\overline\Omega)\) and \(V \in L^1(\Omega)\) such that
\[
K = \{x \in \overline{\Omega} : u(x) = 0\}, 
\]
and the equation
\[
- \Delta u + V u = 0	
\]
is satisfied pointwisely and in the sense of distributions in \(\Omega\).
\end{proposition}

The proof of this proposition requires some minor changes compared to the previous one, which concern mostly what we mean by the \(W^{1, 2}\) capacity being a limit of the \(W^{2, p}\) capacities as \(p\) tends to \(1\).
This should be carefully explained since the \(W^{1, 2}\) capacity and the \(W^{2, 1}\) capacity are not equivalent~\citelist{\cite{MazSha:2009}*{Chapter~1} \cite{Ponce:2013}*{Chapter~16}}. 
The \(W^{2, 1}\) capacity is in fact equivalent to the \(\cH^{N - 2}_{\delta}\) Hausdorff outer measures for any \(0 < \delta < +\infty\). 
As a result, taking a compact set \(K \subset \R^{N}\) whose \(N-2\) dimensional Hausdorff measure satisfies \(0 < \cH^{N - 2}(K) < +\infty\), then one has
\[{}
\capt_{W^{1, 2}}{(K)} = 0
\quad \text{and} \quad
\capt_{W^{2, 1}}{(K)} > 0.
\]

The main issue in the proof of Proposition~\ref{propositionExampleW12} is to make sure that all estimates are given in terms of \(\norm{\Delta\varphi}_{L^1(\Omega)}\) instead of \(\norm{D^2\varphi}_{L^1(\Omega)}\).
The reason is that the capacities associated with the quantities
\[
\int\limits_\Omega \abs{\nabla\varphi}^2
\quad \text{and} \quad
\int\limits_\Omega \abs{\Delta\varphi}
\]
are equal up to a multiplicative constant~\citelist{\cite{Brezis_Marcus_Ponce:2007}*{Theorem~4.E.1}}.
We actually need a weaker property, namely for every compact set \(K \subset \R^{N}\) and for every \(\epsilon > 0\) there exists a nonnegative function \(\varphi \in C_c^\infty(\R^{N})\) such that
\(\varphi > 1\) in a neighborhood of \(K\) and
\[
\norm{\Delta \varphi}_{L^1(\R^{N})}
\le C \capt_{W^{1, 2}}{(K)} + \epsilon,
\]
for some constant \(C > 0\) independent of \(K\)~\citelist{\cite{Ponce:2013}*{Chapter~12}}.
Next, when \(p=1\) the proof of the variant of Maz'ya's inequality (Lemma~\ref{lemmaMazya}) gives the stronger property,
\[
\int\limits_{\R^N} \frac{\abs{\nabla\varphi}^{2}}{(1 + \varphi)^{2}}
\le \int\limits_{\R^N} \frac{\abs{\Delta\varphi}}{1 + \varphi},
\]
and in this case estimate \eqref{eqEstimate2} becomes
\[
\Bignorm{\frac{\abs{\nabla w_j}^2}{w_j^2}}_{L^1(\Omega)}
\le C\bigg(\sum_{i=0}^j \norm{\Delta \varphi_i}_{L^1(\Omega)}^\frac{1}{2} \bigg)^2.
\]
Combining theses modifications, we get the proof of Proposition~\ref{propositionExampleW12} by mimicking the proof of Proposition~\ref{propositionExample}.

As a final remark, it is possible to merge Theorem~\ref{theoremStrongMaximumPrinciple} and its counterpart for \(p = 1\) in a single statement by using a suitable capacity defined in terms of the Laplacian.
Indeed, given a smooth bounded open set \(\Omega \subset \R^{N}\) and a compact set \(K \subset \Omega\), for every \(p \ge 1\) consider
\[
\capt_{\Delta^{p}}{(K; \Omega)}
= \inf{\bigg\{ \norm{\Delta \varphi}_{L^{p}(\Omega)}^p 
: \varphi \in C_c^\infty(\Omega)\ \text{nonnegative}\  \text{and}\ \varphi > 1\ \text{in \(K\)}  \bigg\}}.
\]
This capacity has the same compact sets of zero capacity in \(\Omega\) as \(\capt_{W^{2, p}}\) by the Calderón-Zygmund estimates, while for \(p = 1\) it has the same compact sets of zero capacity in \(\Omega\) as \(\capt_{W^{1, 2}}\).{}
In this respect, we can interpret \(\capt_{W^{1, 2}}\) as the limit of \(\capt_{W^{2, p}}\) as \(p\) tends to \(1\) through this equivalent capacity \(\capt_{\Delta^{p}}\).

\section*{Acknowledgements}
The second author (ACP) was supported by the Fonds de la Recherche scientifique--FNRS under grant number J.0025.13.
He warmly thanks the Dipartimento di Matematica of Sapienza--Universit\`a di Roma for the invitation and hospitality.

%%%%%%%%%%%%%%%%%%%%%%%%%%%%%%%%%%%%%%%%%%%%%%%%%%%%%%%%%%%%%%%%%%%%%%
%%%%%%%%%%%%%%%%%%%%%%%%%%%%%%%%%%%%%%%%%%%%%%%%%%%%%%%%%%%%%%%%%%%%%%
%%%%%%%%%%%%%%%%%%%%%%%%%%%%%%%%%%%%%%%%%%%%%%%%%%%%%%%%%%%%%%%%%%%%%%

\end{document}